\documentclass[final]{siamltex}

\usepackage{color}
\usepackage{cite}
\usepackage{graphicx}
\usepackage{psfrag}
\usepackage{url}
\usepackage{stfloats}
\usepackage{amsmath}
\usepackage{amssymb}
\usepackage{graphicx}
\usepackage{float}
\usepackage{epstopdf}
\usepackage{titlesec}
\titlelabel{\thetitle.\enspace}

\newtheorem{remark}{Remark}

\def\dref#1{(\ref{#1})}

 \def\dfrac{\displaystyle\frac}
\def\be{\begin{equation}}
\def\bel{\begin{equation}\label}
\def\ee{\end{equation}}
\def\ba{\begin{array}}
\def\ea{\end{array}}
\def\banl{\begin{eqnarray}\label}
\def\ean{\end{eqnarray}}

 \def\bna{\begin{eqnarray}}
\def\ena{\end{eqnarray}} \def\dref#1{(\ref{#1})}

\title{IS IT POSSIBLE TO STABILIZE  DISCRETE-TIME PARAMETERIZED UNCERTAIN SYSTEMS GROWING EXPONENTIALLY FAST?\thanks{This work was supported in part by the National Natural Science Foundation of China under Grants 61422308 and 11688101.}}

\author{Zhaobo Liu\thanks{Key Laboratory of Systems and Control, Academy of Mathematics and Systems Science, Chinese Academy of Sciences, 100190, P.~R.~China, and  School of Mathematical Sciences, University of Chinese Academy of Sciences, Beijing 100049, P.~R.~China. Corresponding author: Chanying Li (cyli@amss.ac.cn).}\and
 Chanying Li\footnotemark[2]}

\begin{document}
\maketitle

\begin{abstract}
  This paper derives a somewhat surprising but interesting enough result on  the stabilizability of discrete-time parameterized uncertain systems.
Contrary to an intuition, it shows that the growth rate of a discrete-time stabilizable system  with linear parameterization   is not necessarily to be small all the time. More specifically, to achieve the stabilizability,    the system function $f(x)=O(|x|^b)$ with $b<4$ is only required for a very  tiny fraction of   $x$ in $\mathbb{R}$,     even  if it grows exponentially fast for the other $x$. The proportion of the mentioned  set in $\mathbb{R}$, where the system fulfills the growth rate $ O(|x|^b)$ has also been computed,  for both the stabilizable   and unstabilizable cases. This proportion, as  indicated  herein, could be arbitrarily small, while the corresponding system is stabilizable.
\end{abstract}

\begin{keywords}
  Stochastic adaptive control, feedback limitations, stabilizability, nonlinear systems, discrete-time, least squares
\end{keywords}

\begin{AMS}
  93D15, 93D20, 93E24
\end{AMS}

\pagestyle{myheadings}
\thispagestyle{plain}
\markboth{STABILIZABILITY OF  PARAMETERIZED UNCERTAIN SYSTEMS}{ZHAOBO LIU AND CHANYING LI}

\section{Introduction.}
Linear systems (\cite{Astr95}, \cite{cg91}, \cite{goodwin},
\cite{IS96}) and nonlinear systems with nonlinearities having linear growth rates (\cite{TK}, \cite{xieguo00a}) are studied extensively in adaptive control theory. It is natural that we  keep our mind here concentrating on systems with output nonlinearities growing faster than linearities. Such investigations in the literature are mostly focused on control systems in continuous time (\cite{Marconi}, \cite{Krstic}, \cite{Khalil}). Now comes the noteworthy part. The similarities of adaptive control between continuous- and discrete-time systems  no longer exist.  A large class of  continuous-time nonlinear systems can be globally stabilized by applying
nonlinear damping or back-stepping techniques,  regardless of how fast its growth rate is (\cite{ka} and  \cite{KKK95}). But its discrete-time counterpart is obviously lack of such good property. It was found  early in  \cite{guo97} that fundamental difficulties
arise for adaptive control of discrete-time nonlinear parameterized systems.  \cite{guo97} proved that any feedback control law may fail to stabilize a discrete-time parameterized  system, if  its nonlinearity is  too high.    Such problem also troubles the control of
discrete-time nonparametric nonlinear systems (\cite{xieguo00b},
\cite{Ma08n}, \cite{zhangguo}), semiparametric uncertain systems (\cite{guo2012}, \cite{Sokolov}), linear stochastic systems with
unknown time-varying parameter processes \cite{xueguohuang}, and continuous-time nonlinear systems with sampled-date observations for
 prescribed sampling rates \cite{xueguo}.

All the phenomena suggest that a feedback  has its limit in stabilizing a discrete-time uncertain system. The feedback limit was first characterized by an exponent $b=4$ in \cite{guo97}, where a  discrete-time nonlinear  stochastic system with a scalar parameter was studied:
\begin{equation}\label{sys1}
y_{t+1}=\theta y^b_t+u_t+w_{t+1}.
 \end{equation}
It showed the system is stabilizable if and only if $b<4$. Later on, \cite{xieguo99} confirms the idea of \cite{guo97} on feedback limitations  by providing an ``impossibility theorem'' for the following   multi-parameter uncertain  system:
  \begin{equation}\label{system}
y_{t+1}=\theta_1 y_t^{b_1}+\theta_2 y_t^{b_2}+\cdots+\theta_n
y_t^{b_n}
 +u_t+w_{t+1}.
\end{equation}
A polynomial rule was proposed therein to describe the nonlinear growth rates that  fail all  feedback control laws in stabilizing  system \ref{system}. This  rule   was recently proved to be the necessary and sufficient condition of the
stabilizability of system \dref{system} (see \cite{lilam13}). Note that the systems mentioned above are of linear parametrization. As to the nonlinear parametrization case, some initial research indicates as well that $b=4$ is indeed an important exponent  for the stabilizability of  the underlying uncertain scalar-parameter  systems \cite{lichen16}. Meanwhile, a  parallel theory for the  stabilizability  of discrete-time systems   in the deterministic framework  has  also  been developed  accordingly. The interested readers are referred to \cite{liguo08}, \cite{liguo10}, \cite{liguo11},
\cite{liguo13}, \cite{lichen14}, \cite{lixie06}, \cite{lixieguo06}.

If we consider a model
\begin{equation}\label{sys2}
 y_{t+1}=\theta^T f(y_t)+u_t+w_{t+1},
 \end{equation}
   it is tempting to believe that function $f(x)$ for a stabilizable system should obey the polynomial rule characterized in \cite{xieguo99} (this rule degenerates to $b=4$ when the parameter is of one-dimension), at least
for most $x\in \mathbb{R}$. It may be a little
frustrated that     the polynomial rule forces the largest  value of the exponents $b_1,b_2,\ldots,b_n$  around  $1$, whenever  the number of the unknown parameters are very large. That means, with sufficiently many  parameters, the expected nonlinear growth rate of a stablizable uncertain system in form \ref{sys2} is  close to linear. So, people might guess that discrete-time feedback control has  very limited    capability  in dealing with nonlinear systems.

But, the truth is unexpected. For  the scalar-parameter case, if we denote the set of $x$ that  $f(x)=O(|x|^b)$ with $b<4$  by $S^L_b$,
the results of this paper find that a stabilizable system could   admit of $S^L_b$ being  a very  tiny fraction of   $\mathbb{R}$. How tiny?  As long as the ``proportion'' of  $S^L_b$  in $\mathbb{R}$ does not equal to zero! Roughly speaking, for any $\epsilon>0$,  a scalar-parameter system with
$$
\dfrac{\ell\{x: f(x)\,\, \mbox{grows slower than}\,|x|^b,b<4\}}{\ell\{x\in \mathbb{R}: f(x)\,\, \mbox{grows exponentially} \}}=\epsilon
$$
may still be stabilizable, where  $\ell$ denotes the Lebesgue measure. The least-squares (LS) based self-tuning regulator, as shown later, is competent to  perform the stabilizing task.
This  tells us that a  nonlinear discrete-time parameterized  system,  which grows very fast for most of the time,   still stands a chance to be stabilized by some feedback controller. It is not a surprise that continuous-time controllers  could fulfill such  works, as they regulate systems at every moment.
 However, this is not  obvious in discrete-time control.  There is a certain amount of information loss during controller designs by using sampled date,  especially for a long time running.
 In addition, our results also derive a proportion of $S^L_4$ for  unstabilizable systems.

The rest of the paper is organized as follows. Section \ref{MR} presents our main results on the   stabilizabitlity of a basic class of  nonlinear discrete-time  parameterized systems.
The  proof of the stabilizablity theorem   is contained in  Section  \ref{gsta}, while Section \ref{punsta} treats the unstabilizability part.
The conclusion remarks are finally  given in Section \ref{conre}.

\section{Main results.}\label{MR}
We consider the following parameterized uncertain system
\begin{equation}\label{sys}
 y_{t+1}=\theta f(y_t)+u_t+w_{t+1},~~~~~t\geq0,
 \end{equation}
  where $\theta \in \mathbb{R}$ is an unknown parameter,   $y_t, u_t, w_t \in \mathbb{R}$ are the  system  output, input and noise signals, respectively. Assume
$f: \mathbb{R}\to \mathbb{R}$ is a  known piecewise continuous function and the initial value $y_0$ is independent of $\theta$ and $\{w_t\}$. Moreover,
\begin{description}
\item[A1]
The noise $\{w_t\} $ is an  i.i.d random sequence with  $w_1 \sim N(0,1)$.

\item[A2] Parameter $ \theta\sim N(\theta_0,P_0)$ is independent of $\{w_t\}$.
\end{description}

We begin by studying the stabilizability of system \ref{sys}, which is defined as follows.
\begin{definition}
System \ref{sys} is said to be globally  stabilizable,  if there exits a feedback control law
\begin{equation}\label{utlaw}
u_t\in \mathcal{F}_{t}^y\triangleq\sigma\lbrace y_i,0\leqslant i\leqslant t\rbrace,~t\geq 0
 \end{equation}such that for any initial $y_0\in \mathbb{R},$
$$\sup_{t\geq 1}\frac{1}{t}\sum_{i=1}^{t}y_i^2<+\infty,\quad \mbox{a.s.}.$$
\end{definition}

\subsection{How the proportion of $S_b^L$ in $\mathbb{R}$ matters?}

As already noted   by \cite{guo97},
\begin{equation}\label{grf}
|f(x)|=O(|x|^b)+O(1) \quad\mbox{with}\quad b<4, \,\,x\in \mathbb{R}
 \end{equation}
 is a very important growth rate to guarantee the stabilizability of system \ref{sys}. We might claim that at least  the growth rate \ref{grf} should hold for $x$ in the vast majority of $\mathbb{R}$, or  unstabilizability would be  inevitable. Surprisingly, this is not the case. If we denote
\begin{equation}\label{SbL}
 S_b^L\triangleq\lbrace x:|f(x)|< L|x|^b\rbrace,
  \end{equation}
 the
 ``proportion'' of set $S_b^L$ with $b<4$ in the real number field $\mathbb{R}$ could be arbitrarily small, while system \ref{sys} is still stabilizable, even if $f(x)$  grows exponentially for $x$  outside $S_b^L$. This fact is verified by

\begin{theorem}\label{sta}
Under Assumptions A1--A2,  system \ref{sys} is globally stabilizable if \\
 (i) for some $k_1,k_2>0$,
\begin{equation}\label{fk1k2}
|f(x)|\leq k_1 e^{k_2|x|},\quad \forall x\in \mathbb{R};
\end{equation}
(ii) there exist two numbers  $b<4$ and $L>0$ such that
\begin{equation}\label{SbL>0}
\liminf_{l\to+\infty}\frac{\ell( S_b^L\cap[-l,l])}{l}>0,
\end{equation}
where   $\ell$ denotes the Lebesgue measure.
\end{theorem}

\begin{remark}
Let $p_b\triangleq \liminf_{t\to+\infty} p^l_b$ with $ p^l_b\triangleq   \frac{\ell( S_b^L\cap[-l,l])}{2l}$, then $p_b$ describes the ``proportion'' of  $S_b^L$ in $\mathbb{R}$. Since $p_b>0$ can be taken as small as one likes, Theorem \ref{sta} 
 produces an interesting finding that the growth rate \ref{grf} is only necessary for a   set of $x$ extremely  sparse in $\mathbb{R}$.
\end{remark}

If $f(x)$ grows  no faster than a power function, then   $p^l_b$,     the  proportion    of  $S_b^L$ in interval $[-l,l]$, could converge to zero with a rate $\frac{1}{\log\log l}$ as $l\rightarrow +\infty$ for some properly small $b>0$.

\begin{theorem}\label{sta2}
Under Assumptions A1--A2,  system \ref{sys} is globally stabilizable if \\
(i) for some $a\geq 4$,
\begin{equation}\label{c}
|f(x)|=O(|x|^a)+O(1),\quad \mbox{as}\,\,|x|\rightarrow+\infty;
\end{equation}
 (ii) there exist two numbers  $b<(1+x_{\min})^2$ and $L>0$ such that
\begin{equation}\label{ScL>0}
\liminf_{l\to+\infty}\dfrac{\ell( S_b^L\cap[-l,l])}{l}\cdot\log(\log l)>0,
\end{equation}
where  $x_{\min}$ denotes the smallest solution of equation $x^2-(a-2)x+1=0$.
\end{theorem}

\subsection{What is the proportion of $S^L_b$ for unstabilizable systems?}

Now, we turn to discuss the unstabilizability of system \ref{sys}.

\begin{definition}
System \ref{sys} is  unstabilizable,  if   for any feedback control law $\{u_t\}$ defined by \ref{utlaw}, there exists an initial $y_0$ such that for some    set $D$ with $P(D)>0$,
 $$ \limsup_{t\rightarrow +\infty} \frac{1}{t}\sum_{i=1}^{t}y_i^2=+\infty\quad \mbox{on}\,\, D.$$
\end{definition}

 It is conceivable that   the unstabilizability of system \ref{sys}  depends on  the sparsity of set $S_4^L= \left\lbrace x:|f(x)|< L|x|^4\right\rbrace$ in   $\mathbb{R}$. Indeed, when the
 proportion of set $S_4^L$  in any given interval with length $l$   tends to zero  rapidly  as  $l\rightarrow+\infty$, system \ref{sys} becomes unstabilizable. The required  convergence rate is specified below.

\begin{theorem}\label{unsta}
Under Assumptions A1--A2, system \ref{sys} is unstabilizable  if there exist two numbers $\delta, L>0$ such that as $l\to +\infty$,
$$\sup\limits_{x\in \mathbb{R}}\dfrac{\ell(S_4^L\cap[x-l,x+l])}{l}=O\left(\frac{1}{(\log (\log l))^{1+\delta}}\right).$$
\end{theorem}

\begin{remark}
Note that in Theorem \ref{sta2},  $b<(1+x_{\min})^2=4$ when $a=4$. Then, in view of Theorem \ref{unsta}, if
$|f(x)|=O(|x|^4+1)$, we in fact derive a law of iterated logarithm  $\frac{1}{\log\log l}$   that almost describes the ``critical convergence rate'' of  $p^l_b$ to guarantee the stabilizability of system \ref{sys}.
\end{remark}

Theorem \ref{unsta} can be sharpened.   Assume $h:[0,+\infty)\to[0,+\infty)$ is a nonnegative monotone increasing piecewise continuous function  and satisfies $h(x)=O(x^4)+O(1)$. Let $g(x)\triangleq |x|^{-\frac{1}{4}}h^{-1}(|x|)$, where $h^{-1}$ denotes the inverse function of $h$. Theorem \ref{unsta} is a   direct consequence of  the following theorem  by taking $h(x)=Lx^4$ and
 $g(x)\equiv L^{-\frac{1}{4}}$.

\begin{theorem}\label{unsta1}
Under Assumptions A1--A2, system \ref{sys} is unstabilizable if there is a $\delta>0$ such that
 $$\sup\limits_{x\in \mathbb{R}}\dfrac{\ell(S_h\cap[x-l,x+l])}{l}=O\left(\frac{1}{(\log(\log l))^{1+\delta}}\right) ,$$
 where $S_h\triangleq\lbrace x:|f(x)|< h(|x|)\rbrace$ with $h$  satisfying
\begin{equation}\label{utasi}
\sum_{t=1}^{+\infty}\sup\limits_{x\in [e^{2^{t}},+\infty)}x^{-\frac{1}{16t^2}}g(x)<+\infty.
\end{equation}
\end{theorem}


\section{Proof of Theorem \ref{sta}.}\label{gsta}

In order to prove the stabilizability of system \ref{sys}, we construct a  feedback control law based on the least-squares  (LS) algorithm. The standard LS estimate $\theta_t$ for parameter $\theta$ can be recursively defined by
\begin{eqnarray}\label{LS}
\left\{
\begin{array}{l}
\theta_{t+1}=\theta_t+a_tP_t\phi_t(y_{t+1}-u_t-\phi_t^T\theta_t)\\
P_{t+1}=P_t-a_tP_t\phi_t\phi_t^TP_t,~~P_0>0\\
\phi_t\triangleq{f(y_t)},~~t\geq0
\end{array},
\right.
\end{eqnarray}
where   $a_t\triangleq(1+\phi_t^TP_t\phi_t)^{-1}~$and$~(\theta_0,P_0)~$are the deterministic initial values of the algorithm. According to the ``certainty equivalence principle'', it is a  natural way to design the  stabilizing control by
\begin{equation} \label{ut}
u_t=-\theta_tf(y_t),\quad t\geq0.
\end{equation}

Now, for the closed-loop system \ref{sys}, \dref{LS} and \ref{ut}, one has
\begin{eqnarray}\label{yt}
&&\tilde{\theta_t}=\frac{1}{r_{t-1}}\left\lbrace \tilde{\theta_0}-\sum_{i=0}^{t-1}\phi_i w_{i+1}\right\rbrace,
\nonumber\\
&&y_{t+1}=\tilde{\theta_t}f(y_t)+w_{t+1},
\end{eqnarray}
where $\tilde{\theta_t}\triangleq\theta-\theta_t$, $r_{-1}\triangleq P_0^{-1}$, $r_t\triangleq P_{t+1}^{-1}=P_0^{-1}+\sum_{i=0}^{t}\phi_i^2$, $t\geq0$. Since the LS algorithm \dref{LS} is exactly  the standard Kalman filter for $ \theta\sim N(\theta_0,P_0)$, it yields that $\theta_t=E[\theta|\mathcal{F}_t^y]$ and $P_t=E[(\tilde{\theta_t})^2|\mathcal{F}_t^y]$. Hence,    $y_{t+1}$ possesses a  conditional    Gaussian distribution given $\mathcal{F}_t^y$. For any $t\geq0$, the conditional mean and variance are
\begin{eqnarray}\label{mean}
&&m_t\triangleq{E[y_{t+1}|\mathcal{F}_t^y]}=u_t+\theta_t\phi_t=0,\quad\mbox{a.s.}\\
&&\label{var}
\sigma_t^2\triangleq {Var(y_{t+1}|\mathcal{F}_t^y)}=1+\phi_tP_t\phi_t=\frac{\phi_t^2}{r_{t-1}}+1=\frac{r_t}{r_{t-1}},\quad\mbox{a.s.}
\end{eqnarray}

The proof of Theorem \ref{sta} is prefaced with several technique lemmas. The first presents a very simple fact, which is repeatedly used in the subsequent computations.

\begin{lemma}\label{one}
 If  $\lbrace c_t\rbrace_{t\geq 1}$ satisfies $\liminf_{t\rightarrow+\infty}\frac{c_t}{\log t}>0$, then
\begin{equation}
\sum_{t=1}^{+\infty}\int_{|x|\geqslant c_t}e^{-\frac{x^2}{2}}\,dx<+\infty\nonumber.
\end{equation}
\end{lemma}

\begin{proof}
Since $\liminf_{t\rightarrow+\infty}\frac{c_t}{\log t}>0$ implies that there is a $c>0$ such that for any sufficiently large $t>0$, $c_t>c\log t$, it suffices  to prove $\sum_{t=1}^{+\infty}\int_{|x|\geqslant c\log t}e^{-\frac{x^2}{2}}\,dx<+\infty$.
Note that for $t\geq\max\left\{ e^{\frac{4}{c^2}},2\right\}$,
\begin{eqnarray*}
\int_{|x|\geqslant c\log t}e^{-\frac{x^2}{2}}\,dx&=&\sum_{i=t}^{+\infty}\int_{c\log i}^{c\log(i+1)}e^{-\frac{x^2}{2}}\,dx<\sum_{i=t}^{+\infty}\int_{c\log i}^{c\log(i+1)}e^{-\frac{c^2\log^2 i}{2}}\,dx\nonumber\\
&=&\sum_{i=t}^{+\infty}c\log\left(\frac{i+1}{i}\right)i^{-\frac{c^2\log i}{2}}<\sum_{i=t}^{+\infty}\frac{c}{i^{1+\frac{c^2\log i}{2}}}\\
&\leq&\sum_{i=t}^{+\infty}\frac{c}{i^{3}}<\int_{t-1}^{+\infty}\frac{c}{x^3}\,dx=\frac{c}{2(t-1)^2},\nonumber
\end{eqnarray*}
which leads to Lemma \ref{one} immediately.
\end{proof}

\begin{lemma}\label{two}
For any $n\in \mathbb{Z}^+$, let $\lbrace A_m^n\rbrace_{m\geq1}$ be a sequence of events that $A_m^n\triangleq{\lbrace y_{mn},y_{mn+1},\dots,y_{mn+n-1}\in S_b^L\rbrace}$.
If \ref{SbL>0} holds, then $\sum_{m=1}^{+\infty}I_{A_m^n}=+\infty$ almost surely.
\end{lemma}

\begin{proof}
At first, the piecewise continuity of $f$ infers that $S^L_b$ contains a nonempty interval. Taking a point $\rho$ from this interval,
by \dref{SbL>0},  there exists a  $c_1>0$ such that
$
\inf_{l>0}\frac{S_b^L\cap[\rho-l,\rho+l]}{l}>c_1.
$
Note that $y_{i+1}$ is conditional Gaussian with the conditional mean $m_i=0$ and variance $\sigma_i^2$ by \dref{mean} and \dref{var}, it yields that
\begin{eqnarray*}
P\left(y_{i+1}\in S_b^L|\mathcal{F}_i^y\right)&=&\frac{1}{\sqrt{2\pi}}\int_{|x\sigma_i|\in S_b^L}e^{-\frac{x^2}{2}}\,dx\\
&\geq &\frac{1}{\sqrt{2\pi}}\int_{|x\sigma_i|\in S_b^L,|x-\rho\sigma_{i}^{-1}|\leqslant 1}e^{-\frac{x^2}{2}}\,dx\nonumber\\
&\geq &\ell(x:|x\sigma_i|\in S_b^L,|x-\rho\sigma_{i}^{-1}|\leqslant 1)\cdot \frac{1}{\sqrt{2\pi}}e^{-\frac{(1+|\rho|)^2}{2}}\nonumber\\
&=&\dfrac{\ell(S_b^L\cap [\rho-\sigma_i,\rho+\sigma_i])}{\sigma_i}\cdot \frac{1}{\sqrt{2\pi}}e^{-\frac{(1+|\rho|)^2}{2}}>\frac{c_1}{\sqrt{2\pi}}e^{-\frac{(1+|\rho|)^2}{2}}.\nonumber
\end{eqnarray*}
As a result,  for all $m\geq 1$,
\begin{eqnarray}
P(A_m^n|\mathcal{F}_{mn-1}^y)&=&E\left\lbrace\prod_{i=mn}^{mn+n-1}I_{A_i^1}\bigg|\mathcal{F}_{mn-1}^y\right\rbrace\nonumber\\
&=&E\left\lbrace E[I_{A_{mn+n-1}^{1}}|\mathcal{F}_{mn+n-2}^y]\cdot\prod_{i=mn}^{mn+n-2}I_{A_i^1}\bigg|\mathcal{F}_{mn-1}^y\right\rbrace\nonumber\\
&=&E\Bigg\lbrace P(y_{mn+n-1}\in S_b^L|\mathcal{F}_{mn+n-2}^y)\cdot\prod_{i=mn}^{mn+n-2}I_{A_i^1}\bigg|\mathcal{F}_{mn-1}^y\Bigg\rbrace\nonumber\\
&\geq &\frac{c_1}{\sqrt{2\pi}}e^{-\frac{(1+|\rho|)^2}{2}}\cdot E\left\lbrace\prod_{i=mn}^{mn+n-2}I_{A_i^1}\bigg|\mathcal{F}_{mn-1}^y\right\rbrace\geq \ldots\geq \left(\frac{c_1}{\sqrt{2\pi}}e^{-\frac{(1+|\rho|)^2}{2}}\right)^n.\nonumber
\end{eqnarray}
Consequently, we conclude that $\sum_{m=1}^{+\infty}P(A_m^n|\mathcal{F}_{mn-1}^y)=+\infty$. The lemma then follows  from the \textit{Borel-Cantelli-Levy} theorem.
\end{proof}

\begin{lemma}\label{three}
Denote $S_b(l_1,l_2)\triangleq{\lbrace x: |f(x)|< l_1+l_2|x|^b\rbrace}$ for $b>1$, $l_1\geq0,~l_2>0$. Let $\lbrace B_m\rbrace$ and $\lbrace C_m\rbrace$ be two sequences of the events defined by
\begin{eqnarray}
B_{m+1}&\triangleq &\left\lbrace y_{m+1}\in S_b(l_1,l_2),\sigma_{m+1}^2\geqslant r_m^q,\sigma_m^2\leqslant r_{m-1}^{\frac{1+q}{b-1-q}-\varepsilon}\right\rbrace,\nonumber\\
C_{m+1}&\triangleq &\left\lbrace y_{m+1}\in S_b(l_1,l_2),\sigma_{m+1}^2\geqslant \lambda,\sigma_m^2\leqslant r_{m-1}^{\frac{1}{b-1}-\varepsilon}\right\rbrace,\nonumber
\end{eqnarray}
where  $q\in (0,b-1),~\varepsilon\in(0,1)$ and $\lambda>1$.
If $\liminf_{t\to+\infty}\frac{r_t}{t}>0\,\, a.s.$,
then
\[
\sum_{m=1}^{\infty}I_{B_{m+1}}<+\infty\quad \mbox{and}\quad
\sum_{m=1}^{\infty}I_{C_{m+1}}<+\infty,\quad \mbox{a.s.}.
\]
\end{lemma}
\begin{proof}
Let $A\triangleq\frac{\varepsilon(b-1-q)^2}{b-\varepsilon(b-1-q)}>0$ and
\begin{eqnarray*}
Q_m&\triangleq&{(r_m^A-r_{m}^{A-q})^\frac{1}{2b}(1-l_1(r_m^{1+q}-r_{m})^{-\frac{1}{2}})^{\frac{1}{b}}l_2^{-\frac{1}{b}}}\\
&=&l_2^{-\frac{1}{b}}r_m^{\frac{A}{2b}}(1-r_m^{-q})^{\frac{1}{2b}}(1-l_1(r_m^{1+q}-r_{m})^{-\frac{1}{2}})^{\frac{1}{b}}.
\end{eqnarray*}
Since $|f(y_{m+1})|< l_1+l_2|y_{m+1}|^b$ on set $\lbrace y_{m+1}\in S_b(l_1,l_2)\rbrace$ and $(r_m^{1+q}-r_{m})^{\frac{1}{2}}>l_1$ for all sufficiently large $m$, one has
\begin{eqnarray}\label{qm}
P(B_{m+1}|\mathcal{F}_{m}^y)
&\leq &P\left(|y_{m+1}|\geqslant ((r_m^{1+q}-r_{m})^{\frac{1}{2}}-l_1)^\frac{1}{b}l_2^{-\frac{1}{b}},\sigma_m^2\leqslant r_{m-1}^{\frac{1+q}{b-1-q}-\varepsilon}\Big|\mathcal{F}_{m}^y\right)\nonumber\\
&=&I_{\left\lbrace\sigma_m^2\leqslant r_m^{\frac{1+q-\varepsilon(b-1-q)}{b-\varepsilon(b-1-q)}}\right\rbrace}\cdot E\left\lbrace I_{\left\lbrace |y_{m+1}|\geqslant ((r_m^{1+q}-r_{m})^{\frac{1}{2}}-l_1)^\frac{1}{b}l_2^{-\frac{1}{b}}\right\rbrace}\bigg|\mathcal{F}_{m}^y\right\rbrace\nonumber\\
&=&I_{\left\lbrace\sigma_m^2\leqslant r_m^{\frac{1+q-\varepsilon(b-1-q)}{b-\varepsilon(b-1-q)}}\right\rbrace}\cdot\frac{1}{\sqrt{2\pi}}\int_{|x\cdot\sigma_m|\geqslant ((r_m^{1+q}-r_{m})^{\frac{1}{2}}-l_1)^\frac{1}{b}l_2^{-\frac{1}{b}}}e^{-\frac{x^2}{2}}\,dx\nonumber\\
&\leq &\frac{1}{\sqrt{2\pi}}\int_{\Big|x\cdot r_m^{\frac{1}{2}\frac{1+q-\varepsilon(b-1-q)}{b-\varepsilon(b-1-q)}}\Big|\geqslant ((r_m^{1+q}-r_{m})^{\frac{1}{2}}-l_1)^\frac{1}{b}l_2^{-\frac{1}{b}}}e^{-\frac{x^2}{2}}\,dx\nonumber\\
&=&\frac{1}{\sqrt{2\pi}}\int_{|x|\geqslant Q_m}e^{-\frac{x^2}{2}}\,dx.
\end{eqnarray}
Since $\liminf_{m\to+\infty}\frac{r_m}{m}>0$ implies
$\liminf_{m\to+\infty}\frac{Q_m}{\log m}>0,$
  Lemma \ref{one} shows that $$\sum_{m=1}^{+\infty}\frac{1}{\sqrt{2\pi}}\int_{|x|\geqslant Q_m}e^{-\frac{x^2}{2}}\,dx
<+\infty.$$
Consequently, by \dref{qm}, $\sum_{m=1}^{+\infty}P(B_{m+1}|\mathcal{F}_{m}^y)<+\infty$,
which  leads to $\sum_{m=1}^{\infty}I_{B_{m+1}}<+\infty$, in view of
the \textit{Borel-Cantelli-Levy} theorem.

Let $B\triangleq\frac{\varepsilon(b-1)^2}{b-\varepsilon(b-1)}>0$ and $Q_m^{(1)}\triangleq{((\lambda-1)r_m^B)^\frac{1}{2b}(1-l_1((\lambda-1)r_m)^{-\frac{1}{2}})^{\frac{1}{b}}l_2^{-\frac{1}{b}}}$.
The next claim is treated in a   similar manner  by noting that for all sufficiently large $m$,
\begin{eqnarray}\label{qm1}
P(C_{m+1}|\mathcal{F}_{m}^y)
&\leq &P\left(|y_{m+1}|\geqslant (((\lambda-1)r_m)^{\frac{1}{2}}-l_1)^\frac{1}{b}l_2^{-\frac{1}{b}},\sigma_m^2\leqslant r_{m-1}^{\frac{1}{b-1}-\varepsilon}\Big|\mathcal{F}_{m}^y\right)\nonumber\\
&=&I_{\left\lbrace\sigma_m^2\leqslant r_m^{\frac{1-\varepsilon(b-1)}{b-\varepsilon(b-1)}}\right\rbrace}\cdot\frac{1}{\sqrt{2\pi}}\int_{|x\cdot\sigma_m|\geqslant (((\lambda-1)r_m)^{\frac{1}{2}}-l_1)^\frac{1}{b}l_2^{-\frac{1}{b}}}e^{-\frac{x^2}{2}}\,dx\nonumber\\
&\leq &\frac{1}{\sqrt{2\pi}}\int_{\big|x\cdot r_m^{\frac{1}{2}\frac{1-\varepsilon(b-1)}{b-\varepsilon(b-1)}}\big|\geqslant (((\lambda-1)r_m)^{\frac{1}{2}}-l_1)^\frac{1}{b}l_2^{-\frac{1}{b}}}e^{-\frac{x^2}{2}}\,dx\nonumber\\
&=&\frac{1}{\sqrt{2\pi}}\int_{|x|\geqslant Q_m^{(1)}}e^{-\frac{x^2}{2}}\,dx.
\end{eqnarray}
Then, by  $\liminf_{m\to+\infty}\frac{Q_m^{(1)}}{\log m}>0$ and  Lemma \ref{one}, $$\sum_{m=1}^{+\infty}\frac{1}{\sqrt{2\pi}}\int_{|x|\geqslant Q_m^{(1)}}e^{-\frac{x^2}{2}}\,dx
<+\infty,$$
which, together with \dref{qm1}, implies $\sum_{m=1}^{+\infty}P(C_{m+1}|\mathcal{F}_{m}^y)<+\infty$.
 Finally,  with probability $1$, $\sum_{m=1}^{\infty}I_{C_{m+1}}<+\infty$  according to the \textit{Borel-Cantelli-Levy} theorem again.
\end{proof}

\begin{lemma}\label{five}
If $\ell(\lbrace x:|f(x)|>0\rbrace)>0$, then $\liminf_{t\to+\infty}\frac{r_t}{t}>0$ almost surely.
\end{lemma}
\begin{proof} For any $ c>0$,  denote $T_c\triangleq{\left\lbrace x\in \mathbb{R}:|f(x)|> c\right\rbrace}$.
Since $\lbrace x:|f(x)|>0\rbrace=\bigcup_{n=1}^{+\infty}T_{\frac{1}{n}}$, there exists an integer $n\geq 1$ such that $\ell\left(T_{\frac{1}{n}}\right)>0$. Moreover, $\cup_{i=1}^{+\infty}(T_{\frac{1}{n}}\cap [-i,i])=T_{\frac{1}{n}}\cap\left(\cup_{i=1}^{+\infty}[-i,i]\right)=T_{\frac{1}{n}}$,
which shows that there is an  $i\geq 1$ satisfying $\ell\left(T_{\frac{1}{n}}\cap[-i,i]\right)>0$.
For the two integers $n$ and $i$ defined above, denote
\begin{eqnarray}
d\triangleq {\ell\left(T_{\frac{1}{n}}\cap [-i,i]\right)},\quad E_{m+1}\triangleq {\left\lbrace y_{m+1}\in T_{\frac{1}{n}},\sigma_m^2<2\right\rbrace},\quad m\geq 1.\nonumber
\end{eqnarray}
We estimate the conditional probability of $E_{m+1}$ for each $m\geq 1$ by
\begin{eqnarray}\label{em}
P(E_{m+1}|\mathcal{F}_{m}^y)
&=&I_{\left\lbrace \sigma_{m}^2<2\right\rbrace}\cdot \frac{1}{\sqrt{2\pi}}\int_{|x\cdot\sigma_{m}|\in T_{\frac{1}{n}}}e^{-\frac{x^2}{2}}\,dx\nonumber\\
&\geq &I_{\left\lbrace \sigma_{m}^2<2\right\rbrace}\cdot\frac{1}{\sqrt{2\pi}}\int_{|x\cdot\sigma_{m}|\in T_{\frac{1}{n}}\cap{[-i,i]}}e^{-\frac{x^2}{2}}\,dx\geq I_{\left\lbrace\sigma_{m}^2<2\right\rbrace}\cdot \frac{1}{\sqrt{2\pi}}\frac{d}{\sigma_{m}}{e^{-\frac{i^2}{2\sigma_{m}^2}}}\nonumber\\
&\geq &I_{\left\lbrace\sigma_{m}^2<2\right\rbrace}\cdot \frac{1}{\sqrt{2\pi}}\frac{d}{\sigma_{m}}{e^{-\frac{i^2}{2}}}\geq c_2 I_{\left\lbrace \sigma_{m}^2<2\right\rbrace},
\end{eqnarray}
where $c_2\triangleq{\frac{1}{\sqrt{2\pi}}\frac{d}{\sqrt{2}}e^{-\frac{i^2}{2}}}\in(0,1)$.
Next, for each $m\geq 1$, denote
\begin{eqnarray}\label{fm}
F_{m+1}&\triangleq&{\left\lbrace \sigma_{m}^2\geqslant 2\right\rbrace\cup\left\lbrace y_{m+1}\in T_{\frac{1}{n}}\right\rbrace}=\left\lbrace\sigma_{m}^2\geqslant 2\right\rbrace \cup E_{m+1},
\end{eqnarray}
which, together with \dref{em},  leads to
\begin{eqnarray}\label{pfm}
P(F_{m+1}|\mathcal{F}_{m}^y)=E\left\lbrace I_{\left\lbrace \sigma_{m}^2\geqslant 2\right\rbrace}+I_{E_{m+1}}\Big|\mathcal{F}_{m}^y\right\rbrace&\geq &I_{\left\lbrace \sigma_{m}^2\geqslant 2\right\rbrace}+c_2\cdot I_{\left\lbrace \sigma_{m}^2<2\right\rbrace}\geq c_2.
\end{eqnarray}

Now, set $x_{m}\triangleq I_{F_{m}}-E[I_{F_{m}}|\mathcal{F}_{m-1}^y]$,
and it is clear that $\sup_{m\geq 1}E\left\lbrace x^2_{m}|\mathcal{F}_{m-1}^y\right\rbrace <+\infty$.
Since $\lbrace x_m,\mathcal{F}_{m}^y\rbrace_{m\geq1}$ forms a martingale difference sequence, by  applying the strong law of  large numbers for the   martingale differences, one has $\sum_{m=1}^{t}x_m=o(t)$ almost surely. Therefore,  by \dref{pfm}, for all sufficiently large $t$,
\begin{eqnarray}
\frac{\sum_{m=1}^{t}I_{F_m}}{t}=o(1)+\frac{\sum_{m=1}^{t}P(F_{m}|\mathcal{F}_{m-1}^y)}{t}>\dfrac{c_2}{2},\nonumber
\end{eqnarray}
 and hence
 $\sum_{m=1}^{t}I_{F_m}>\frac{c_2}{2}t$ almost surely.  As a consequence,
\begin{eqnarray}
\max\left\lbrace \sum_{m=1}^{t}I_{\left\lbrace \sigma_{m}^2\geqslant 2\right\rbrace},\sum_{m=1}^{t}I_{\left\lbrace y_m\in T_{\frac{1}{n}}\right\rbrace} \right\rbrace&\geq &\frac{1}{2}\left(\sum_{m=1}^{t}I_{\left\lbrace \sigma_{m}^2\geqslant 2\right\rbrace}+\sum_{m=1}^{t}I_{\left\lbrace y_m\in T_{\frac{1}{n}}\right\rbrace}\right)\nonumber\\
&=&\frac{1}{2}\sum_{m=1}^{t}I_{F_m}>\frac{c_2}{4}t,\quad\mbox{a.s.}.\nonumber
\end{eqnarray}

We complete the remainder of the proof by considering the following two cases:\\
\emph{Case 1:} $\sum_{m=1}^{t}I_{\lbrace\sigma_{m}^2\geqslant 2\rbrace}>\frac{c_2}{4}t.$
This means the events $\lbrace\sigma_{m}^2\geqslant 2\rbrace$, $1\leq m\leq t$,  occur at least $\lceil\frac{c_2}{4}t\rceil$ times, and hence $r_t=r_0\cdot\prod_{m=1}^t\sigma_{m}^2\geq 2^{\frac{c_2}{4}t}P_0^{-1}$.\\
\emph{Case 2:} $\sum_{m=1}^{t}I_{\lbrace y_m\in T_{\frac{1}{n}}\rbrace}>\frac{c_2}{4}t.$
Then, $\lbrace y_m\in T_{\frac{1}{n}}\rbrace$, $1\leq m\leq t$,  occur at least $\lceil\frac{c_2}{4}t\rceil$ times.  Since $\left\lbrace y_m\in T_{\frac{1}{n}}\right\rbrace=\lbrace f^2(y_m)> \frac{1}{n^2}\rbrace$,
one has $r_t=P_0^{-1}+\sum_{m=0}^{t}f^2(y_m)>\frac{c_2}{4n^2}t$.

Finally, by combining Case 1 and Case 2, it  shows that  $r_t\geq \min\left\lbrace 2^{\frac{c_2}{4}t}P_0^{-1},\frac{c_2}{4n^2}t\right\rbrace$ for all sufficiently large $t$ almost surely,
which proves the lemma.
\end{proof}

\begin{lemma}\label{four}
Let   \dref{fk1k2} hold for some  $k_1,k_2> 0$. If there is a  set $D$ with $P(D)>0$ such that
$\sup\limits_{t}\sigma_t=+\infty$  on $D$  and $\liminf_{t\to+\infty}\frac{r_t}{t}>0$  almost surely, then $$\lim_{t\rightarrow +\infty}\sigma_t=+\infty,\quad \mbox{a.s.}\quad \mbox{on}\,\,D.$$
\end{lemma}
\begin{proof} Given a number $z>1$,  define $D_{m+1}\triangleq{\left\lbrace\sigma_{m}^2\leqslant z,\sigma_{m+1}^2\geqslant z\right\rbrace}$ and $Q_{m}^{(2)}\triangleq{\frac{1}{2k_2\sqrt{z}}\log(\frac{(z-1)r_m}{k_1^2})}$, $m\geq 0$.
Therefore,
\begin{eqnarray}\label{qm2}
P(D_{m+1}|\mathcal{F}_{m}^y)&=&P\left(f^2(y_{m+1})\geqslant (z-1)r_m,\sigma_{m}^2\leqslant z\Big|\mathcal{F}_{m}^y\right)\nonumber\\
&\leq &P\left(|y_{m+1}|\geqslant \frac{1}{k_2}\log\left(\frac{\sqrt{(z-1)r_m}}{k_1}\right),\sigma_{m}^2\leqslant z\bigg|\mathcal{F}_{m}^y\right)\nonumber\\
&=&I_{\left\lbrace \sigma_{m}^2\leqslant z\right\rbrace}\cdot \frac{1}{\sqrt{2\pi}}\int_{|x\cdot\sigma_{m}|\geqslant \frac{1}{2k_2}\log\left(\frac{(z-1)r_m}{k_1^2}\right)}e^{-\frac{x^2}{2}}\,dx\nonumber\\
&\leq &\frac{1}{\sqrt{2\pi}}\int_{|x\cdot\sqrt{z}|\geqslant \frac{1}{2k_2}\log\left(\frac{(z-1)r_m}{k_1^2}\right)}e^{-\frac{x^2}{2}}\,dx=\frac{1}{\sqrt{2\pi}}\int_{|x|\geqslant Q_{m}^{(2)}}e^{-\frac{x^2}{2}}\,dx.
\end{eqnarray}

Since $\liminf_{m\to+\infty}\frac{r_m}{m}>0$ shows $\liminf_{m\to+\infty}\frac{Q_{m}^{(2)}}{\log m}>0$, by
Lemma \ref{one}  and  \dref{qm2},
$$\sum_{m=1}^{+\infty}P(D_{m+1}|\mathcal{F}_{m}^y)\leq \sum_{m=1}^{+\infty}\frac{1}{\sqrt{2\pi}}\int_{|x|\geqslant Q_m^{(2)}}e^{-\frac{x^2}{2}}\,dx<+\infty,\quad  \mbox{a.s.}.$$
Taking account of the \textit{Borel-Cantelli-Levy} theorem, events $\lbrace D_{m}\rbrace$ occur only finite times  almost surely.
Therefore, if $\sigma_{m}^2\leqslant z$ for some  sufficiently large $m$ on a set $D'\subset D$ with $P(D')>0$, then $\sigma_{m+1}^2\leqslant z$, and hence $\sigma_{t}^2\leqslant z$ for all $t\geq m+2$ with probability $P(D')>0$. That is, $\sup_{t}\sigma_{t}<+\infty$ almost surely on $D'$, which  contradicts  to the assumption that $\sup_{t}\sigma_{t}=+\infty$ a.s. on $D$.    Hence, for any $z>1$, with probability $P(D)$,  there is a random $m>0$ such that     $\sigma_{t}^2>z$ for all $t\geq m$. This is exactly
$\lim_{t\rightarrow +\infty}\sigma_{t}=+\infty$ a.s. on $D$ by taking $z$ over all the natural numbers.
\end{proof}

\begin{lemma}\label{six}
For some  $a_0\geq 0$ and $\varepsilon_i\in (0,\frac{1}{i+1})$, define a sequence $\lbrace a_i\rbrace$ by
\begin{equation}\label{ai}
a_{i+1}=\frac{1+a_i}{b-1-a_i}-\varepsilon_i,\quad i\geq 0.
\end{equation}
(i) If $b\in(1,4)$ and $a_0=0$, then there exists a positive integer $k$ and a sequence $\lbrace\varepsilon_i\rbrace_{i=0}^{k-1}$, such that $a_i\in (0,b-1)$ for $1\leq i\leq k-1$ and $a_k>b-1$.\\
(ii) If $b\geq 4$ and $x_1<a_0<b-1$, where $x_1$ is the maximal real solution of equation $x^2-(b-2)x+1=0$, then there exists a positive integer $k$ and a sequence $\lbrace\varepsilon_i\rbrace_{i=0}^{k-1}$, such that $a_i\in (x_1,b-1)$ for $1\leq i\leq k-1$ and $a_k>b-1$.\\
(iii) If $b\geq 4$ and $a_0=0$, then there exists a sequence $\lbrace\varepsilon_i\rbrace_{i\geq 0}$, such that $\lim_{i\rightarrow +\infty}a_i=x_2$ and $a_i<a_{i+1}<x_2$ for all $i\geq0$, where $x_2$ is the minimum real solution of equation $x^2-(b-2)x+1=0$
\end{lemma}

\begin{proof}
(i) Since $b\in (1,4)$, it is clear that
$a_0^2-(b-2)a_0+1>0$, and hence
\begin{eqnarray}\label{de}
\frac{1+a_0}{b-1-a_0}>a_0.
\end{eqnarray}
Note that $a_0=0<b-1$, there exists some $\varepsilon'_0\in(0,1)$ such that
\begin{equation}\label{a1}
a'_{1}\triangleq\frac{1+a_0}{b-1-a_0}-\varepsilon'_0>a_0=0.
\end{equation}

We now prove  assertion (i) by reduction to absurdity. Suppose
it is not true, then $a'_1\in (0,b-1]$. Obviously, \dref{ai} means any
\begin{eqnarray}\label{e0}
\varepsilon_0\in \left(\varepsilon'_0,\min\left\{ \frac{1+a_0}{b-1-a_0}-a_0, 1  \right\}\right)
\end{eqnarray}
achieves $a_1\in (0,b-1)$.
If for some $k\geq 1$,
there is a sequence $\lbrace\varepsilon_i\rbrace_{i=0}^{k-1}$ with $\varepsilon_i\in (0,\frac{1}{i+1})$   such that $a_i\in (0,b-1)$ for all $i\in [1, k]$,
according to the same arguments for \dref{de} and \dref{a1}, one has $a'_{k+1}\triangleq\frac{1+a_k}{b-1-a_k}-\varepsilon'_k>a_k>0$  for some $0<\varepsilon'_k <\frac{1}{k+1}$. Take  a $\varepsilon_k \in (0,\frac{1}{k+1})$ in a similar way of \dref{e0},  then the corresponding $a_{k+1}\in (0,b-1)$ in view of the  hypothesis that (i) does not hold.
This means, by induction, there exists a sequence $\lbrace \varepsilon_i\rbrace_{i=0}^{+\infty}$  with   $\varepsilon_i\in(0,\frac{1}{i+1})$   such that $0<a_i<a_{i+1}<b-1$ for all $i\geq 1$.
Hence, $\lim_{i\rightarrow+\infty}a_i$ exists and $\lim_{i\rightarrow+\infty}\varepsilon_i=0$.

Denote $a\triangleq \lim_{i\rightarrow+\infty}a_i$, then
\begin{eqnarray}\label{asolution}
a=\lim_{i\rightarrow+\infty}a_{i+1}=\lim_{i\rightarrow+\infty}\frac{1+a_i}{b-1-a_i}-\lim_{i\rightarrow+\infty}\varepsilon_i&=&\dfrac{1+a}{b-1-a}.
\end{eqnarray}
Therefore,  $a\in (0,b-1)$ and it serves as a solution of equation $x^2-(b-2)x+1=0$, which is impossible due to $b<4$.\\
(ii) The proof is almost the same as that of (i),  by noting that $x_1<b-1$ and
  for any $a_i\in(x_1,b-1)$,  $b\geq4 $ yields    $\frac{1+a_i}{b-1-a_i}>a_i$.
  As a matter of fact, if the assertion fails,  one can take a series of $\lbrace \varepsilon_i\rbrace_{i=0}^{+\infty}$  with   $\varepsilon_i\in(0,\frac{1}{i+1})$ such that $x_1<a_i<a_{i+1}<b-1$ for all $i\geq 0$.   There thus arises a contradiction between $a\triangleq\lim_{i\rightarrow+\infty}a_i\geq a_0>x_1$ and $a^2-(b-2)a+1=0$ by \dref{asolution} with $b\geq 4$.
\\
(iii) With $b\geq 4$,  any  $a_i<x_2$ implies  that there is a  $\varepsilon_{i}\in(0,\frac{1}{i+1})$ such that $a_{i+1}=\frac{1+a_i}{b-1-a_i}-\varepsilon_i>a_i$.
Since $a_i<x_2$ also leads to
$
\frac{1+a_i}{b-1-a_i}<x_2,\nonumber
$
\begin{equation}\label{ai+1<1}
a_i<a_{i+1}=\frac{1+a_i}{b-1-a_i}-\varepsilon_i<x_2.
 \end{equation}

Noting that  $a_0=0$,  \dref{ai+1<1} yields  that $a_0<a_1<x_2$ for some $\varepsilon_0\in (0,1)$. By induction, there is a
  sequence $\lbrace\varepsilon_i\rbrace_{i\geq 0}$ satisfying $\lim_{i\rightarrow+\infty}\varepsilon_i=0$  such that \dref{ai+1<1} holds for all  $i\geq0$. Thus, $\lim_{i\rightarrow+\infty}a_i$ exists. Letting $a\triangleq{\lim_{i\rightarrow+\infty}a_i}$ shows   that $a=\frac{1+a}{b-1-a}$ by \dref{asolution}, and hence $a=x_2$.
\end{proof}

\begin{lemma}\label{r/r<}
Let system \dref{sys} satisfy \dref{SbL>0} and  $\ell(\lbrace x:|f(x)|>0\rbrace)>0$, then
\begin{equation}\label{supr}
\sup\limits_{t} \sigma_{t}<+\infty,\quad \mbox{a.s.}.\nonumber
\end{equation}
\end{lemma}

\begin{proof}
We only need to consider the case $b>1$.
Since $\ell(\lbrace x:|f(x)|>0\rbrace)>0$,
Lemma \ref{five} shows $\liminf_{t\to+\infty}\frac{r_t}{t}>0$ a.s.. Assume $D\triangleq \{\sup\limits_{t} \sigma_{t}=+\infty\}$ satisfies $P(D)>0$,
 by Lemma \ref{four}, $\lim_{t\rightarrow +\infty}\sigma_{t}=+\infty$ on $D$ almost surely.
 That is, for any $\lambda>1,$ there is a random $N>0$ such that
\begin{equation}\label{rlam}
\sigma_t^2>\lambda\quad\mbox{if}\quad  t>N,\quad\mbox{a.s.}\quad \mbox{on}\,\,D.
\end{equation}

Now,  according to (i) of Lemma \ref{six}, for some integer $k\geq 1$,  one can construct a finite sequence $\lbrace a_i\rbrace_{i=1}^k$  satisfying $0=a_0<\ldots<a_{k-1}<b-1$, $a_k>b-1$ and
\begin{equation}\label{aj}
a_{i+1}=\frac{1+a_i}{b-1-a_i}-\varepsilon_i,
\end{equation}
where $\varepsilon_i\in (0,1)$ for $0\leq i\leq k-1$. Fix this $k$.
Applying  Lemma \ref{two} indicates that   $\sum_{m=1}^{+\infty}I_{A_m^{k+1}}=+\infty$ a.s. on $D$, that is, $\lbrace y_{m(k+1)},y_{m(k+1)+1},\ldots,y_{m(k+1)+k}\in S_b^L\rbrace$ occurs  infinitely many  times for $m$ on $D$. Let
\begin{equation}\label{ytk}
T\triangleq \{t_j:y_{t_j(k+1)},y_{t_j(k+1)+1},\ldots,y_{t_j(k+1)+k}\in S_b^L\}.
\end{equation}
Clearly, $|T|=\aleph_0$ on $D$ almost surely. In view of
 \dref{rlam} and \dref{ytk}, as long as $t_j\in T$
is sufficiently large,
\begin{equation}\label{ytk1}
 y_{t_j(k+1)+k}\in S_b^L\quad \mbox{and}\quad\sigma_{t_j(k+1)+k}^2>\lambda\quad\mbox{a.s.}\quad\mbox{on}\,\,D.
\end{equation}

Now, we use the induction method to prove that if $t_j\in T$ is  sufficiently large,
\begin{eqnarray}\label{ytk2}
\left\{
\begin{array}{l}
 y_{t_j(k+1)+k-i}\in S_b^L \\
 \sigma_{t_j(k+1)+k-i}^2>r_{t_j(k+1)+k-i-1}^{a_i}
 \end{array}
 \right.
 \quad\mbox{a.s.}\quad\mbox{on}\,\, D
\end{eqnarray}
holds   for all $1\leq i\leq k$, where $a_i$ is defined by \dref{aj}.
In fact, for $i=1$, Lemma \ref{three} with $l_1=0$, $l_2=L$ and $\varepsilon=\varepsilon_0$, yields $\sum_{m=1}^{\infty}I_{C_{m+1}}<+\infty$ a.s., which  indicates that the events ${\lbrace y_{m+1}\in S_b^L,\sigma_{m+1}^2\geqslant\lambda,\sigma_m^2\leqslant r_{m-1}^{\frac{1}{b-1}-\varepsilon_0}\rbrace}$
occur only  finite times for $m$. This, together with \dref{aj} and
\dref{ytk1}, leads to
\begin{eqnarray}
\sigma_{t_j(k+1)+k-1}^2&>&r_{t_j(k+1)+k-2}^{\frac{1}{b-1}-\varepsilon_0}=r_{t_j(k+1)+k-2}^{a_1},\quad\mbox{a.s.}\quad\mbox{on}\,\,D,\nonumber
\end{eqnarray}
when $t_j\in T$ is sufficiently large. Thus,  \dref{ytk2} holds for $i=1$  due to  \dref{ytk}.

Assume  that \dref{ytk2} is true  for some $i=s\in [1,k)$, when $t_j\in T$  is sufficiently large.
By Lemma \ref{three} again with $l_1=0$, $l_2=L$, $q=a_i$ and $\varepsilon=\varepsilon_i$, one has $\sum_{m=1}^{\infty}I_{B_{m+1}}<+\infty$. Hence, events $\lbrace y_{m+1}\in S_b^L,\sigma_{m+1}^2\geqslant r_m^{a_i},\sigma_m^2\leqslant r_{m-1}^{\frac{1+a_i}{b-1-a_i}-\varepsilon_i}\rbrace$
occur finite times for  $m$.   By the hypothesis,  for the sufficiently large $t_j$,
\begin{eqnarray}
\sigma_{t_j(k+1)+k-s-1}^2&>&r_{t_j(k+1)+k-s-2}^{\frac{1+a_s}{b-1-a_s}-\varepsilon_s}=r_{t_j(k+1)+k-s-2}^{a_{s+1}}, \quad\mbox{a.s.}\quad\mbox{on}\,\,D,\nonumber
\end{eqnarray}
and hence  \dref{ytk2} also holds for $i=s+1$ in view of \dref{ytk}. The assertion is thus proved.

Now, it follows immediately from \dref{ytk2} that for all sufficiently large $t_j \in T$,
 $$ y_{t_j(k+1)}\in S_b^L \quad \mbox{and}\quad \sigma_{t_j(k+1)}^2>r_{t_j(k+1)-1}^{a_k},\quad\mbox{a.s.}\quad\mbox{on}\,\, D.$$
 Denote $G_{m+1}\triangleq\lbrace y_{m+1}\in S_b^L,\sigma_{m+1}^2>r_{m}^{a_k}\rbrace$, then $\sum_{m=1}^{+\infty}I_{G_{m+1}}=+\infty,~\mbox{a.s.}$ on $D$,
as $|T|=\aleph_0$ almost surely.
This means $P\left(D\right)\leq P(\sum_{m=1}^{+\infty}I_{G_{m+1}}=+\infty)$.
Lemma \ref{r/r<} thus becomes straightforward if     we could prove
\begin{equation}\label{pigm}
P\left(\sum_{m=1}^{+\infty}I_{G_{m+1}}=+\infty\right)=0,
\end{equation}
which contradicts to the hypothesis that $P(D)>0$.

To this end, for each sufficiently large $m$, compute
\begin{eqnarray}\label{pgm}
P(G_{m+1}|\mathcal{F}_{m}^y)&=&P\left(y_{m+1}\in S_b^L,f^2(y_{m+1})>r_{m}^{1+a_k}-r_{m}|\mathcal{F}_{m}^y\right)\nonumber\\
&\leq& P\left(y_{m+1}\in S_b^L,|y_{m+1}|>(r_{m}^{1+a_k}-r_{m})^{\frac{1}{2b}}L^{-\frac{1}{b}}\Big|\mathcal{F}_{m}^y\right)\nonumber\\
&\leq& \frac{1}{\sqrt{2\pi}}\int_{|x\cdot\sigma_{m}|\geqslant(r_{m}^{1+a_k}-r_{m})^{\frac{1}{2b}}L^{-\frac{1}{b}}}e^{-\frac{x^2}{2}}\,dx\nonumber\\
&\leq &\frac{1}{\sqrt{2\pi}}\int_{|x\cdot\sqrt{P_0r_{m}}|\geqslant(r_{m}^{1+a_k}-r_{m})^{\frac{1}{2b}}L^{-\frac{1}{b}}}e^{-\frac{x^2}{2}}\,dx=\frac{1}{\sqrt{2\pi}}\int_{|x|\geqslant Q_{m}^{(3)}}e^{-\frac{x^2}{2}}\,dx,
\end{eqnarray}
where $Q_m^{(3)}\triangleq P_0^{-\frac{1}{2}}L^{-\frac{1}{b}}(r_m^{a_k+1-b}-r_m^{1-b})^{\frac{1}{2b}}$.
Since $a_k>b-1$ and   $\liminf_{m\rightarrow +\infty}\frac{r_m}{m}>0$, one has $\liminf_{m\rightarrow +\infty}\frac{Q_m^{(3)}}{\log m}>0$.  By virtue of Lemma \ref{one}, $$\sum_{m=1}^{+\infty}\frac{1}{\sqrt{2\pi}}\int_{|x|\geqslant Q_m^{(3)}}e^{-\frac{x^2}{2}}\,dx
<+\infty.$$
So, \dref{pgm} yields that $\sum_{m=1}^{+\infty}P(G_{m+1}|\mathcal{F}_{m}^y)<+\infty$.
This shows $\sum_{m=1}^{\infty}I_{G_{m+1}}<+\infty$ a.s.
by the \textit{Borel-Cantelli-Levy} theorem and \dref{pigm} follows as desired.
\end{proof}

We are now in a position  to prove Theorem \ref{sta}.

\begin{proof}[Proof of Theorem \ref{sta}]
First of all, if $\ell(\lbrace x:|f(x)|>0\rbrace)=0$, then for any $i\geq 0$,
$
E\left[I_{\lbrace f(y_{i+1})\neq 0\rbrace}|\mathcal{F}_{i}^y\right]=P(f(y_{i+1})\neq 0|\mathcal{F}_{i}^y)=0$,
which yields
\begin{equation}
E\left[\sum_{i=0}^{+\infty}I_{\lbrace f(y_{i+1})\neq 0\rbrace}\right]=\sum_{i=0}^{+\infty}E\lbrace E[I_{\lbrace f(y_{i+1})\neq 0\rbrace}|\mathcal{F}_{i}^y]\rbrace=0.\nonumber
\end{equation}
This infers that $\sum_{i=0}^{+\infty}I_{\lbrace f(y_{i+1})\neq 0\rbrace}=0,~\mbox{a.s.}.$
That is,  with probability $1$, $f(y_{i+1})= 0$ for all $i\geq 0$.  By \dref{ut}, $u_{i+1}\equiv 0, i\geq 0$.   Then, system \dref{sys} reduces to
\begin{equation}\label{yw}
y_{i+2}=w_{i+2},\quad i\geq 0\quad\mbox{a.s.},\nonumber
\end{equation}
and the stabilizability is verified by
$\sum_{i=1}^{t}y_{i}^2=y_1^2+\sum_{i=2}^{t}w_{i}^2=O(t),~ \mbox{as}~t\rightarrow+\infty.$

Therefore, it is sufficient to consider the case where  $\ell(\lbrace x:|f(x)|>0\rbrace)>0$. Taking account of Lemma \ref{r/r<},
\begin{equation}\label{supr1}
\sup\limits_{t} \sigma_{t}<+\infty,\quad \mbox{a.s.}.
\end{equation}
Moreover, recall from \cite[Lemma 3.1]{guo95}  that $\sum_{i=0}^{t}\alpha_i=O(\log r_t)$ a.s.,
where $\alpha_i\triangleq\frac{(\phi_i\tilde{\theta_i})^2}{1+\phi_iP_i\phi_i}$,  \dref{yt} and \dref{supr1} lead to
\begin{eqnarray}\label{logr}
\sum_{i=0}^{t}(y_{i+1}-w_{i+1})^2&=&\sum_{i=0}^{t}(\phi_i\tilde{\theta_i})^2=\sum_{i=0}^{t}\alpha_i\frac{r_i}{r_{i-1}}=O(\log r_t).
\end{eqnarray}
Observe that
\begin{eqnarray}\label{logr2}
\log r_t&\leq &\log\left(P_0^{-1}+(t+1)\max \limits_{0\leqslant i\leqslant t}f^2(y_i)\right)\leq\log\left(P_0^{-1}+(t+1)k_1^2 e^{2k_2\max \limits_{0\leqslant i\leqslant t}|y_i|}\right)\nonumber\\
&=&O(1)+O(\log t)+O\left(\max \limits_{0\leqslant i\leqslant t}|y_i|\right)=O(1)+O(\log t)+O\left(\bigg(\sum_{i=0}^{t}y_i^2\bigg)^{\frac{1}{2}}\right)\nonumber
\end{eqnarray}
and $\sum_{i=0}^{t}2[(y_{i+1}-w_{i+1})^2+w_{i+1}^2]\geq \sum_{i=0}^{t}y_{i+1}^2$,
 \dref{logr}  becomes
\begin{eqnarray}
\frac{1}{2}\sum_{i=0}^{t}y_{i}^2 
&\leq &\sum_{i=0}^{t}(y_{i+1}-w_{i+1})^2+ \sum_{i=0}^{t}w_{i+1}^2+\frac{1}{2}y_0^2=O(\log r_t)+\sum_{i=0}^{t}w_{i+1}^2\nonumber\\
&\leq &O(1)+O(\log t)+O\left(\bigg(\sum_{i=0}^{t}y_i^2\bigg)^{\frac{1}{2}}\right)+\sum_{i=0}^{t}w_{i+1}^2,\quad \mbox{a.s.}.\nonumber
\end{eqnarray}
Moreover, since $\sum_{i=0}^{t}w_{i+1}^2=O(t)$,  as $t\rightarrow+\infty$,
\begin{equation}
\frac{1}{2t}\sum_{i=0}^{t}y_{i}^2=O(1)+O\left(\frac{\log t}{t}\right)+o\left(\frac{1}{t}\sum_{i=0}^{t}y_{i}^2\right)^{\frac{1}{2}},\nonumber
\end{equation}
which is exactly $\frac{1}{t}\sum_{i=0}^{t}y_{i}^2=O(1)$ almost surely.
\end{proof}

\section{Proof of  unstabilizability.}\label{punsta}
Because  Theorem \ref{unsta} is an immediate  corollary of Theorem \ref{unsta1}, we only  provide the proof of Theorem  \ref{unsta1}  here.

\begin{proof}[Proof of Theorem  \ref{unsta1}]
Denote
\begin{eqnarray}
\left\{
\begin{array}{l}
H_0\triangleq\left\lbrace\omega: r_0> e^2,y_0\not\in S_h\right\rbrace,\nonumber\\
H_t\triangleq\left\lbrace\omega:r_t> r_{t-1}^{2+\frac{1}{t}},y_t\not\in S_h \right\rbrace,~~~~~~t\geq1\nonumber\\
\end{array}
\right.
\end{eqnarray}
 and $H\triangleq\bigcap_{t=0}^{+\infty} H_t$. The following argument is mainly devoted to  verifying $P(H)>0$.

Now,
 for any $t\geq 1$,
\begin{eqnarray}\label{jkl}
P(H_{t+1}^c|\mathcal{F}_{t}^y)&=&P\left(\left\lbrace r_{t+1}\leqslant r_{t}^{2+\frac{1}{t+1}},y_{t+1}\not\in S_h \right\rbrace\cup\lbrace y_{t+1}\in S_h \rbrace\Big|\mathcal{F}_{t}^y\right)\nonumber\\
&=&P\left(f^2(y_{t+1})\leqslant r_{t}^{2+\frac{1}{t+1}}-r_t,y_{t+1}\not\in S_h \Big|\mathcal{F}_{t}^y\right)+P(y_{t+1}\in S_h|\mathcal{F}_{t}^y).
\end{eqnarray}
To compute this probability, observe that by the assumption of the theorem, there is a $k_3>0$ such that for  any $ x\in \mathbb{R}$,
\begin{equation}\label{s4l}
\ell(S_h\cap[x-l,x+l])\leq \frac{k_3 l}{(\log(\log l))^{1+\delta}},\quad l\geq 3.
\end{equation}
Denote $J_t \triangleq \frac{2}{\sqrt{2\pi}}\sigma_{t}^{-1}r_{t}^{\frac{1}{4}\left(1+\frac{1}{2(t+1)}\right)}g\Big(r_{t}^{1+\frac{1}{2(t+1)}}\Big)$,
then in view of \dref{mean} and \dref{var},
\begin{eqnarray}\label{jkl1}
&&P\left(f^2(y_{t+1})\leqslant r_{t}^{2+\frac{1}{t+1}}-r_t,y_{t+1}\not\in S_h |\mathcal{F}_{t}^y\right)\nonumber\\
&\leq &P\left(h^2(|y_{t+1}|)\leqslant r_{t}^{2+\frac{1}{t+1}},y_{t+1}\not\in S_h |\mathcal{F}_{t}^y\right)\nonumber\\
&\leq &\frac{1}{\sqrt{2\pi}}\int_{|m_t+x\cdot\sigma_t|\leqslant r_{t}^{\frac{1}{4}\left(1+\frac{1}{2(t+1)}\right)}g\Big( r_{t}^{1+\frac{1}{2(t+1)}}\Big)}e^{-\frac{x^2}{2}}\,dx
\leq  J_t,
\end{eqnarray}
Furthermore,
by \dref{s4l} and letting $K_t\triangleq\frac{1}{\sqrt{2\pi}}\int_{|x|>3\log(t+e)}e^{-\frac{x^2}{2}}\,dx$,
it follows that
\begin{eqnarray}\label{jkl2}
P(y_{t+1}\in S_h|\mathcal{F}_{t}^y)-K_t
&=&\frac{1}{\sqrt{2\pi}}\int_{m_t+x\cdot \sigma_t\in S_h,|x|\leqslant 3\log(t+e)}e^{-\frac{x^2}{2}}\,dx\nonumber\\
&\leq &\frac{\ell(S_h\cap[m_t-3\sigma_t \log(t+e),m_t+3\sigma_t \log(t+e)])}{\sqrt{2\pi}\sigma_t}\nonumber\\
&\leq &\frac{3k_3\log(t+e)}{\sqrt{2\pi}\bigg(\log\log\Big(3\sigma_t\log(t+e)\Big)\bigg)^{1+\delta}}
\triangleq L_t.
\end{eqnarray}
Since
 $r_t\geqslant e^{2^{t+1}},\sigma_t^2>e^{2^{t}}$ on $\bigcap_{i=0}^{t} H_i$, it derives
 \begin{eqnarray}\label{Hisigma}
 \bigcap_{i=0}^{t} H_i\subseteq \left\lbrace r_t> r_{t-1}^{2+\frac{1}{t}},r_t\geqslant e^{2^{t+1}},\sigma_t^2>e^{2^{t}}\right\rbrace,
 \end{eqnarray}
   and hence
\begin{eqnarray}\label{ih}
\prod_{i=1}^{t}I_{H_i}&=&I_{\left\lbrace r_t> r_{t-1}^{2+\frac{1}{t}},r_t\geqslant e^{2^{t+1}}\right\rbrace}\cdot\prod_{i=1}^{t}I_{H_i}=I_{\left\lbrace \sigma_t^2>e^{2^{t}}\right\rbrace}\cdot\prod_{i=1}^{t}I_{H_i}.
\end{eqnarray}
Substituting   \dref{jkl1} and  \dref{jkl2} into \dref{jkl}, one has $P(H_{t+1}^c|\mathcal{F}_{t}^y)\leq J_t+K_t+L_t$ and by \dref{ih},
\begin{eqnarray}\label{pht}
P\left(\bigcap_{i=0}^{t+1} H_i\right)&=&E\left\lbrace E[I_{H_{t+1}}|\mathcal{F}_{t}^y]\prod_{i=0}^{t}I_{H_i}\right\rbrace\geq E\left\lbrace(1-J_t-K_t-L_t)\prod_{i=0}^{t}I_{H_i}\right\rbrace\nonumber\\
&=&E\Bigg\lbrace\prod_{i=0}^{t}I_{H_i}-J_t \cdot I_{\left\lbrace r_t> r_{t-1}^{2+\frac{1}{t}},r_t\geqslant e^{2^{t+1}}\right\rbrace}\prod_{i=0}^{t}I_{H_i}-K_t\prod_{i=0}^{t}I_{H_i}\nonumber\\
&&-L_t\cdot I_{\left\lbrace \sigma_t^2>e^{2^{t}}\right\rbrace}\prod_{i=0}^{t}I_{H_i}\Bigg\rbrace.
\end{eqnarray}
At the same time,
\begin{eqnarray}\label{jti}
J_t\cdot I_{\left\lbrace r_t> r_{t-1}^{2+\frac{1}{t}},r_t\geqslant e^{2^{t+1}}\right\rbrace}&=&\frac{2}{\sqrt{2\pi}}\frac{r_{t}^{\frac{1}{4}\left(\frac{2t+3}{2(t+1)}\right)}g\Big( r_{t}^{1+\frac{1}{2(t+1)}}\Big)}{\sigma_t}\cdot I_{\left\lbrace \sigma_t^2> r_{t}^{\frac{1+t}{1+2t}},r_t\geqslant e^{2^{t+1}}\right\rbrace}\nonumber\\
&\leq &\frac{2}{\sqrt{2\pi}}\frac{r_{t}^{\frac{1}{4}\left(\frac{2t+3}{2(t+1)}\right)}g\Big(r_{t}^{1+\frac{1}{2(t+1)}}\Big)}{\sqrt{r_{t}^{\frac{1+t}{1+2t}}}}\cdot I_{\left\lbrace\sigma_t^2> r_{t}^{\frac{1+t}{1+2t}},r_t\geqslant e^{2^{t+1}}\right\rbrace}\nonumber\\
&=&\frac{2}{\sqrt{2\pi}}\dfrac{g\Big(r_{t}^{1+\frac{1}{2(t+1)}}\Big)}{r_t^{\frac{1}{8(1+t)(1+2t)}}}\cdot I_{\left\lbrace r_t\geqslant e^{2^{t+1}},\sigma_t^2> r_{t}^{\frac{1+t}{1+2t}}\right\rbrace}\nonumber\\
&\leq &\frac{2}{\sqrt{2\pi}}\dfrac{g\Big(r_{t}^{1+\frac{1}{2(t+1)}}\Big)}{r_t^{\left(1+\frac{1}{2(t+1)}\right)\frac{1}{16(t+1)^2}}}\cdot I_{\left\lbrace r_t> r_{t-1}^{2+\frac{1}{t}},r_t\geqslant e^{2^{t+1}}\right\rbrace}\nonumber\\
&\leq &\frac{2}{\sqrt{2\pi}}{\sup\limits_{x\in[e^{2^{t+1}},+\infty)}x^{-\frac{1}{16(t+1)^2}}g(x)}\cdot I_{\left\lbrace r_t> r_{t-1}^{2+\frac{1}{t}},r_t\geqslant e^{2^{t+1}}\right\rbrace}.
\end{eqnarray}

Denote
\begin{eqnarray}\label{jt1}
&&J_t^{(1)}\triangleq \frac{2}{\sqrt{2\pi}}{\sup\limits_{x\in[e^{2^{t+1}},+\infty)}x^{-\frac{1}{16(t+1)^2}}g(x)},\label{lt1}\\
&&L_t^{(1)}\triangleq \frac{3k_3\log(t+e)}{\sqrt{2\pi}(\log\log(3e^{2^{t-1}}\log(t+e)))^{1+\delta}},
\end{eqnarray}
then $\sum_{t=1}^{+\infty}J_t^{(1)}<+\infty$   because of  \dref{utasi} and
\begin{eqnarray}\label{Lt1}
L_t\leq L_t^{(1)}=O\left(\frac{\log t}{t^{1+\delta}}\right)\quad \mbox{on}\quad {\left\lbrace \sigma_t^2> e^{2^{t}}\right\rbrace}
\end{eqnarray}
due to \dref{jkl2} that
\begin{eqnarray*}\label{lti}
L_t\cdot I_{\left\lbrace \sigma_t^2> e^{2^{t}}\right\rbrace}
\leq \frac{3k_3\log(t+e)}{\sqrt{2\pi}(\log\log(3e^{2^{t-1}}\log(t+e)))^{1+\delta}}\cdot I_{\left\lbrace \sigma_t^2> e^{2^{t}}\right\rbrace}.\quad
\end{eqnarray*}
In addition,
Lemma \ref{one}  yields $\sum_{t=1}^{+\infty}K_t<+\infty$, hence
\begin{equation}\label{sumjkl}
\sum_{t=1}^{+\infty}(J_t^{(1)}+K_t+L_t^{(1)})<+\infty,\quad \mbox{a.s.}.
\end{equation}
Applying \dref{jti}--\dref{Lt1}, \dref{pht} reduces to
\begin{eqnarray}
P\left(\bigcap_{t=0}^{t+1} H_t\right)&\geq &E\Bigg\lbrace \prod_{i=0}^{t}I_{H_i}-J_t^{(1)}\cdot I_{\left\lbrace r_t> r_{t-1}^{2+\frac{1}{t}},r_t\geqslant e^{2^{t+1}}\right\rbrace}\prod_{i=0}^{t}I_{H_i}-K_t\prod_{i=0}^{t}I_{H_i}\nonumber\\
&&-L_t^{(1)}\cdot I_{\left\lbrace \sigma_t^2>e^{2^{t}}\right\rbrace}\prod_{i=0}^{t}I_{H_i}\Bigg\rbrace=\left(1-J_t^{(1)}-K_t-L_t^{(1)}\right)P\Bigg(\bigcap_{t=0}^{t} H_t\Bigg).\nonumber
\end{eqnarray}
Since \dref{sumjkl} implies that there is a $N_1>0$ such that for all $t\geq N_1$, $J_t^{(1)}+K_t+L_t^{(1)}<1$, it follows that
\begin{eqnarray}
P\Bigg(\bigcap_{t=0}^{+\infty} H_t\Bigg)&=&\lim_{t=N_1}^{+\infty}P\Bigg(\bigcap_{j=0}^{t+1} H_j\Bigg)\geq \lim_{t=N_1}^{+\infty}\Bigg(\prod_{i=N_1}^t\left(1-J_i^{(1)}-K_i-L_i^{(1)}\right)\Bigg) P\Bigg(\bigcap_{j=0}^{N_1} H_j\Bigg)\nonumber\\
&=&\prod_{i=N_1}^{+\infty}\left(1-J_i^{(1)}-K_i-L_i^{(1)}\right)P\Bigg(\bigcap_{j=0}^{N_1} H_j\Bigg)>0,\nonumber
\end{eqnarray}
which is exactly $P(H)>0$.

On the other hand, by \dref{Hisigma}, $\sigma_t^2>e^{2^t},\forall t\geq 1,$ on $H$.
Therefore,  for any $C>0$,
\begin{eqnarray}
&&P\Bigg(\Bigg(\bigcap_{i=0}^{t} H_i\Bigg)\cap\lbrace |y_{t+1}|<C\rbrace\bigg|\mathcal{F}_{t}^y\Bigg)=I_{\bigcap_{i=0}^{t} H_i}\cdot P(|y_{t+1}|<C|\mathcal{F}_{t}^y)\nonumber\\
&=&I_{\bigcap_{i=0}^{t} H_i}\cdot  \frac{1}{\sqrt{2\pi}}\int_{|m_t+x\cdot\sigma_t|<C}e^{-\frac{x^2}{2}}\,dx\leq I_{\left\lbrace\sigma_t^2>e^{2^t}\right\rbrace}\cdot \frac{1}{\sqrt{2\pi}}\frac{2C}{\sigma_t}\leq \frac{1}{\sqrt{2\pi}}\frac{2C}{e^{2^{t-1}}},\nonumber
\end{eqnarray}
and hence
\begin{equation}
\sum_{t=1}^{+\infty}P\Bigg(\Bigg(\bigcap_{i=0}^{t} H_i\Bigg)\cap\lbrace |y_{t+1}|<C\rbrace\bigg|\mathcal{F}_{t}^y\Bigg)<+\infty\quad\mbox{on}\,\,H.\nonumber
\end{equation}
Invoking the \textit{Borel-Cantelli-Levy} theorem, one has
\begin{eqnarray}\label{last}
&&\sum_{t=1}^{+\infty}I_{\bigcap_{i=0}^{t} H_i}\cdot I_{\lbrace |y_{t+1}|<C\rbrace}=\sum_{t=1}^{+\infty}I_{\lbrace |y_{t+1}|<C\rbrace\cap\left(\bigcap_{i=0}^{t} H_i\right)}
<+\infty,\quad \mbox{a.s.}\quad\mbox{on}\,\,H.
\end{eqnarray}
Note that $I_{\bigcap_{i=0}^{t}H_i}=1$ on $H$ for every $ t\geq 1$, $\sum_{t=1}^{+\infty}I_{\lbrace |y_{t+1}|<C\rbrace}<+\infty$
almost surely on $H$, in view of \dref{last}.  This infers that $\liminf_{t\rightarrow +\infty}|y_{t}|\geq C$ on $H$, and consequently, $\lim_{t\rightarrow +\infty}|y_t|=+\infty$ on $H$  by   letting $C\rightarrow +\infty$. So, considering $P(H)>0$,
\begin{equation}
\frac{1}{t}\sum_{i=1}^{t}y_i^2\rightarrow\infty\quad\mbox{as}\quad t\rightarrow +\infty,\quad \mbox{on}\,\,H,\nonumber
\end{equation}
establishing the result.
\end{proof}

\section{Concluding remarks.}\label{conre}

In the very beginning,   the work was intended to seek a connection between the measure of $S^L_b$ and the stabilizability of stochastic parameterized systems in discrete time. But the finding is interesting, as it turns out. It suggests that a discrete-time control law is also capable to deal with high nonlinearity. This paper, of course, is just a starting point to provide some preliminary results for the scalar-parameter case. It calls for further investigations on this topic.

%





\appendix
\section{Proof of Theorem \ref{sta2}}\label{AppA}
This appendix is addressed to prove  Theorem \ref{sta2}. Some lemmas are necessary.

\begin{lemma}\label{seven}
Let $x_{\min}\leq x_{\max}$ denote the two solutions of equation $x^2-(a-2)x+1=0$ and
$\ell(\lbrace x:|f(x)|>0\rbrace)>0$. Under (i) of Theorem \ref{sta2}, then \\
(i) $D_1=D_2$ with $D_1\triangleq {\left\lbrace\sup\limits_{t}\sigma_t=+\infty\right\rbrace}$ and $ D_2\triangleq{\left\lbrace\liminf\limits_{t\rightarrow +\infty}\frac{\log r_t}{\log r_{t-1}}\geq 1+x_{\min}\right\rbrace}$;\\
(ii)  $P(D_3)=0$ with $D_3\triangleq {\left\lbrace\limsup_{t\rightarrow +\infty}\frac{\log r_t}{\log r_{t-1}}>1+x_{\max}\right\rbrace}$.
\end{lemma}

\begin{proof}
To prove (i), note that $\ell(\lbrace x:|f(x)|>0\rbrace)>0$ implies $\liminf_{t\to+\infty}\frac{r_t}{t}>0$ almost surely  by Lemma \ref{five}, so it is enough to show $D_1\subseteq D_2$.
As a matter of fact, in view of Lemma \ref{four}, $\lim_{t\rightarrow +\infty}\sigma_t=+\infty$ on $D_1$ almost surely.
That is, for any $\lambda>1,$ there is a random $N>0$ such that
\begin{equation}\label{rtb}
\sigma_t^2>\lambda\quad\mbox{if}\quad  t>N,\quad\mbox{a.s.}\quad \mbox{on}\,\,D_1.
\end{equation}
Furthermore, according to (iii) of Lemma \ref{six}, there is an infinite sequence $\lbrace a_n\rbrace_{n\geq 0}$ satisfying $0=a_0<\ldots<a_{n}<\ldots<x_{\min}$, $\lim_{n\rightarrow +\infty}a_n=x_{\min}$ and
\begin{equation}\label{an}
a_{n+1}=\frac{1+a_n}{a-1-a_n}-\varepsilon_n,
\end{equation}
where $\varepsilon_n\in (0,1)$ for all $n\geq0$.
 We  use the induction method to prove that for each $ n\geq 1$,  when $m$ is sufficiently large,
\begin{equation}\label{induc}
\sigma_m^2>r_{m-1}^{a_n},\quad\mbox{a.s.}\quad\mbox{on}\,\, D_1.
\end{equation}

Observe that $|f(x)|< l_1+l_2 |x|^a$ for some $l_1,l_2>0$, where  $x\in\mathbb{R}$.
When $n=1$, Lemma \ref{three} with  $\varepsilon=\varepsilon_0$  indicates that the events ${\lbrace\sigma_{m+1}^2\geqslant\lambda,\sigma_m^2\leqslant r_{m-1}^{\frac{1}{a-1}-\varepsilon_0}\rbrace}$
occur  finite times for $m$. Together with \dref{rtb}, it infers that for all sufficiently large $m$, $\sigma_m^2> r_{m-1}^{\frac{1}{a-1}-\varepsilon_0}=r_{m-1}^{a_1}$ almost surely. Now, assume that \dref{induc} holds for some $n\geq 1$, whenever $m$  is sufficiently large. We  prove it  for $n+1$.
Applying Lemma \ref{three} again with $\varepsilon=\varepsilon_n$ and $q=a_n$, events $\lbrace \sigma_{m+1}^2\geqslant r_m^{a_n},\sigma_m^2\leqslant r_{m-1}^{\frac{1+a_n}{a-1-a_n}-\varepsilon_n}\rbrace$
occur  finite times for $m$ as well. So,  for all sufficiently large $m$,
\begin{equation}
\sigma_m^2> r_{m-1}^{\frac{1+a_n}{a-1-a_n}-\varepsilon_n}=r_{m-1}^{a_{n+1}},\quad\mbox{a.s.}\quad\mbox{on}\,\,D_1.\nonumber
\end{equation}
We thus in fact have verified  \dref{induc}  for all $n$, when  $m$ is sufficiently large. This means
\begin{equation}
\liminf_{t\rightarrow +\infty}\frac{\log r_t}{\log r_{t-1}}\geq 1+a_n,\quad\forall n\geq 1,\quad\mbox{a.s.}\quad\mbox{on}\,\,D_1,\nonumber
\end{equation}
and by letting $n\rightarrow +\infty$,
\begin{equation}\label{liminf}
\liminf_{t\rightarrow +\infty}\frac{\log r_t}{\log r_{t-1}}\geq 1+x_{\min},\quad\mbox{a.s.}\quad\mbox{on}\,\,D_1.
\end{equation}

Next,  we show (ii).   For each integer $l\geq 2$,  denote $a_{0,l}\triangleq x_{\max}+l^{-1}(a-1-x_{\max})\in(x_{\max},a-1)$.  According to (ii) of Lemma \ref{six},  there exists a finite sequence $\lbrace a_{i,l}\rbrace_{i=1}^{k_l}$ with some integer $k_l\geq1$ depending on $a_{0,l}$
satisfying $x_{\max}<a_{0,l}<\ldots<a_{k_l-1,l}<a-1,~~a_{k_l,l}>a-1$, and
\begin{equation}
a_{i+1,l}=\frac{1+a_{i,l}}{a-1-a_{i,l}}-\epsilon_{i,l},\nonumber
\end{equation}
where $\epsilon_{i,l}\in(0,1)$ for $0\leq i\leq k_l-1$.
Similar to the induction argument of \dref{induc}, we can prove that for all $l\geq 2$ and $0\leq n\leq k_l$, $$\lbrace\sigma_m^2>r_{m-1}^{a_{0,l}},\,\,\mbox{i.o.}\rbrace\subset\lbrace\sigma_m^2>r_{m-1}^{a_{n,l}},\,\,\mbox{i.o.}\rbrace.$$
Suppose $P(D_3)>0$. Then, there is a random number $a_0$ taking values  from    $\lbrace a_{0,l}, l\in \mathbb{N}^+  \backslash \{1\} \rbrace$ such that  $\sigma_{m+1}^2> r_{m}^{a_0}$ for infinitely many  $m$   on $D_3$. Hence $$P(D_3)\leq \sum_{l=2}^{\infty}P(\sigma_m^2>r_{m-1}^{a_{0,l}},\,\,\mbox{i.o.})\leq \sum_{l=2}^{\infty}P(\sigma_m^2>r_{m-1}^{a_{k_l,l}},\,\,\mbox{i.o.}),$$
and $P(D_3)=0$  will hold if we could show that
\begin{equation}\label{bk0}
P\left(\sigma_m^2>r_{m-1}^{a_{k_l,l}},\,\,\mbox{i.o.}\right)=0,\quad \forall l\geq 2.
\end{equation}

Indeed, for any  $a_{k_l,l}$ and sufficiently large $m$,  $(r_{m}^{1+a_{k_l,l}}-r_{m})^{\frac{1}{2}}-l_1>0$
and
\begin{eqnarray}\label{pbk}
P\left(\sigma_{m+1}^2>r_{m}^{a_{k_l,l}}\Big|\mathcal{F}_{m}^y\right)&=&P\left(f^2(y_{m+1})>r_{m}^{1+a_{k_l,l}}-r_{m}|\mathcal{F}_{m}^y\right)\nonumber\\
&\leq &P\left(|y_{m+1}|>\gamma_2^{-\frac{1}{a}}((r_{m}^{1+a_{k_l,l}}-r_{m})^{\frac{1}{2}}-\gamma_1)^{\frac{1}{a}}\Big|\mathcal{F}_{m}^y\right)\nonumber\\
&\leq &\frac{1}{\sqrt{2\pi}}\int_{|x\cdot\sigma_m|\geqslant \gamma_2^{-\frac{1}{a}}((r_{m}^{1+a_{k_l,l}}-r_{m})^{\frac{1}{2}}-\gamma_1)^{\frac{1}{a}}}e^{-\frac{x^2}{2}}\,dx\nonumber\\
&\leq &\frac{1}{\sqrt{2\pi}}\int_{|x\cdot\sqrt{P_0r_{m}}|\geqslant\gamma_2^{-\frac{1}{a}}((r_{m}^{1+a_{k_l,l}}-r_{m})^{\frac{1}{2}}-\gamma_1)^{\frac{1}{a}}}e^{-\frac{x^2}{2}}\,dx\nonumber\\
&=&\frac{1}{\sqrt{2\pi}}\int_{|x|\geqslant Q_{m}^{(4)}}e^{-\frac{x^2}{2}}\,dx,
\end{eqnarray}
where $Q_m^{(4)}\triangleq P_0^{-\frac{1}{2}}\gamma_2^{-\frac{1}{a}}(r_m^{a_{k_l,l}+1-a}-r_m^{1-a})^{\frac{1}{2a}}(1-\gamma_1(r_{m}^{1+a_{k_l,l}}-r_{m})^{-\frac{1}{2}})^{\frac{1}{a}}$.
Furthermore, since $\liminf_{m\rightarrow +\infty}\frac{r_m}{m}>0$ and $a_{k_l,l}>a-1$, one has  $\liminf_{m\rightarrow +\infty}\frac{Q_m^{(4)}}{\log m}>0$. By Lemma \ref{one},  $\sum_{m=1}^{+\infty}\frac{1}{\sqrt{2\pi}}\int_{|x|\geqslant Q_m^{(4)}}e^{-\frac{x^2}{2}}\,dx<+\infty$,
and hence \dref{pbk} leads to $$\sum_{m=1}^{+\infty}P\left(\sigma_{m+1}^2>r_{m}^{a_{k_l,l}}|\mathcal{F}_{m}^y\right)<+\infty.$$
So, the  \textit{Borel-Cantelli-Levy} theorem yields
\dref{bk0} and  $P(D_3)=0$ follows.
\end{proof}

\begin{proof}[Proof of Theorem \ref{sta2}]
Suppose $P(D_1)>0$, where $D_1$ is defined in Lemma \ref{seven}. It suffices to consider Theorem \ref{sta2} for $\ell(\lbrace x:|f(x)|>0\rbrace)>0$. Since $b<(1+x_{\min})^2$ and $x_{\min}x_{\max}=1$, there exist some $\delta_1,\delta_2\in(0,x_{\min})$ such that
\begin{equation}\label{1j}
\frac{1+x_{\min}-\delta_2}{\frac{x_{\max}+\delta_1}{1+x_{\max}+\delta_1}}>b.
\end{equation}
By Lemma \ref{seven},   for all sufficiently large $m$,
\begin{equation}\label{21}
\sigma_{m+1}^2>r_m^{x_{\min}-\delta_2},\quad\sigma_{m}^2<r_{m-1}^{x_{\max}+\delta_1},\quad\mbox{a.s.}\quad\mbox{on}\,\,D_1.
\end{equation}
However, by letting $Q_{m}^{(5)}\triangleq{L^{-\frac{1}{b}}((r_{m}^{1+x_{\min}-\delta_2}-r_{m})^{\frac{1}{2}})^{\frac{1}{b}}}\cdot r_{m}^{-\frac{1}{2}\frac{x_{\max}+\delta_1}{1+x_{\max}+\delta_1}}$,
\begin{eqnarray}
&&P\left( y_{m+1}\in S_b^{L}, \sigma_{m+1}^2>r_m^{x_{\min}-\delta_2},\sigma_{m}^2<r_{m-1}^{x_{\max}+\delta_1}\right)\nonumber\\
&= &P\left( y_{m+1}\in S_b^{L}, f^2(y_{m+1})>r_m^{1+x_{\min}-\delta_2}-r_m,\sigma_{m}^2<r_{m-1}^{x_{\max}+\delta_1}\right)\nonumber\\
&\leq & I_{\lbrace\sigma_{m}^2<r_{m-1}^{x_{\max}+\delta_1}\rbrace}\cdot\frac{1}{\sqrt{2\pi}}\int_{|x\cdot\sigma_{m}|\geqslant L^{-\frac{1}{b}}((r_{m}^{1+x_{\min}-\delta_2}-r_{m})^{\frac{1}{2}})^{\frac{1}{b}}}e^{-\frac{x^2}{2}}\,dx\nonumber\\
&\leq &\frac{1}{\sqrt{2\pi}}\int_{|x|\geqslant Q_{m}^{(5)}}e^{-\frac{x^2}{2}}\,dx.\nonumber
\end{eqnarray}
Since \dref{1j} and $\liminf_{m\rightarrow +\infty}\frac{r_m}{m}>0$ implies $\liminf_{m\rightarrow +\infty}\frac{Q_m^{(5)}}{\log m}>0$,
by Lemma \ref{one} and the \textit{Borel-Cantelli-Levy} theorem, events $\lbrace y_{m+1}\in S_b^{L}, \sigma_{m+1}^2>r_m^{x_{\min}-\delta_2},\sigma_{m}^2<r_{m-1}^{x_{\max}+\delta_1}\rbrace$ occur finite times for $m$. This fact  together with \dref{21} yileds that  $y_{m+1}\not\in S_{b}^{L}$ for all sufficiently large $m$   on $D_1$ almost surely.
Then,
\begin{equation}\label{ds}
\sum_{m=1}^{+\infty}P(y_{m+1}\in S_{b}^{L}|\mathcal{F}_{m}^{y})<+\infty,\quad\mbox{a.s.}\quad\mbox{on}\,\,D_1.
\end{equation}
On the other hand, Lemma \ref{seven} implies
\begin{equation}
\sup_{m\geq 1}\frac{\log r_{m}}{\log r_{m-1}}<+\infty,\quad\mbox{a.s.}\quad\mbox{on}\,\,D_1,\nonumber
\end{equation}
thus there is a random number $M_1>1$ such that for all $m\geq 1$,
\begin{equation}
\log \sigma_{m}\leq \log r_{m}<M_1^{m},\quad\mbox{a.s.}\quad\mbox{on}\,\,D_1.\nonumber
\end{equation}
Moreover, Lemma \ref{four} yields that $\lim_{m\rightarrow+\infty}\sigma_m=+\infty$ on $D_1$. So, there is a positive constant $M_2$ such that for all sufficiently large $m$,
\begin{equation}
\frac{\ell(S_b^L\cap [-\sigma_m,\sigma_m])}{\sigma_m}\geq \frac{M_2}{\log(\log \sigma_m)}\geq \frac{M_2}{m\log M_1}\quad \mbox{on} \,\,D_1.\nonumber
\end{equation}
As a result,
\begin{eqnarray}
\sum_{m=1}^{+\infty}P(y_{m+1}\in S_{b}^{L}|\mathcal{F}_{m}^{y})&=&\frac{1}{\sqrt{2\pi}}\sum_{m=1}^{+\infty}\int_{|x\sigma_m|\in S_b^L}e^{-\frac{x^2}{2}}\,dx\nonumber\\
&\geq &\frac{1}{\sqrt{2\pi}}\sum_{m=1}^{+\infty}\int_{|x\sigma_m|\in S_b^L,|x|\leqslant 1}e^{-\frac{x^2}{2}}\,dx\nonumber\\
&\geq &\sum_{m=1}^{+\infty}\ell(x:|x\sigma_m|\in S_b^L,|x|\leqslant 1)\cdot \frac{1}{\sqrt{2\pi}}e^{-\frac{1}{2}}\nonumber\\
&=&\sum_{m=1}^{+\infty}\dfrac{\ell(S_b^L\cap [-\sigma_m,\sigma_m])}{\sigma_m}\cdot \frac{1}{\sqrt{2\pi}}e^{-\frac{1}{2}}.\nonumber\\
&\geq &\frac{1}{\sqrt{2\pi}}e^{-\frac{1}{2}}\cdot\frac{M_2}{M_1}\sum_{m=1}^{+\infty}\frac{1}{m}=+\infty\quad \mbox{on}\,\, D_1,\nonumber
\end{eqnarray}
which contradicts to \dref{ds}. So,  $P(D_1)=0$, that is,  $\sup_{t}\sigma_t<+\infty$ almost surely.  The remainder of the proof  is thus  similar to that of Theorem \ref{sta}.
\end{proof}

\bibliographystyle{siamplain}
\bibliography{references}

\end{document}